\documentclass[reqno,12pt]{amsart}
\usepackage{amsmath,amsthm,amssymb,amsfonts,amscd}
\usepackage{epsfig}
\usepackage[all]{xy}
\usepackage{tabularx}  
\usepackage{booktabs}  
\usepackage{float}
\usepackage{mathtext}
\usepackage[T2A]{fontenc}
\usepackage{amssymb}

\usepackage{multicol}

\usepackage[all]{xy} 
\usepackage{quiver}
\usepackage{dirtytalk}

\usepackage[thinlines]{easytable}

\usepackage{graphicx,xcolor}
\graphicspath{{Pictures/}}
\DeclareGraphicsExtensions{.pdf,.png,.jpg}

\numberwithin{equation}{section}

\theoremstyle{plain}
\newtheorem{theorem}{Theorem}

\newtheorem{lemma}[theorem]{Lemma}
\newtheorem{corollary}[theorem]{Corollary}
\newtheorem{proposition}[theorem]{Proposition}
\newtheorem*{theorem*}{Theorem}
\newtheorem*{conjecture*}{Conjecture}

\theoremstyle{definition}
\newtheorem{remark}[theorem]{Remark}
\newtheorem{example}[theorem]{Example}
\newtheorem{definition}[theorem]{Definition}
\newcommand\DistTo{\xrightarrow{
   \,\smash{\raisebox{-0.65ex}{\ensuremath{\scriptstyle\sim}}}\,}}

\newcommand{\CC}{{\mathbb{C}}}

\newcommand{\ZZ}{{\mathbb{Z}}}

\newcommand{\NN}{{\mathbb{N}}}

\newcommand{\Jac}{\mathrm{Jac}}

\newcommand{\id}{\mathrm{id}}

\newcommand{\bx}{{\bf x}}
\newcommand{\by}{{\bf y}}
\newcommand{\bz}{{\bf z}}

\newcommand{\rmH}{{{\rm H}}}

\newcommand{\ccHH}{{\mathsf{HH}}}


\begin{document}

\title[Non-invertible singularities and their Landau-Ginzburg orbifolds]{Non-invertible quasihomogeneous singularities and their Landau-Ginzburg orbifolds}  
\date{\today}
\author{Anton Rarovskii}
\address{\newline Faculty of Mathematics, National Research University Higher School of Economics, Usacheva str., 6, 119048 Moscow, Russian Federation, and \newline
Skolkovo Institute of Science and Technology, Nobelya str., 3, 121205 Moscow, Russian Federation}
\email{aararovskiy@edu.hse.ru}
\begin{abstract}
     According to the classification of quasihomogeneous singularities, any polynomial $f$ defining such singularity has a decomposition $f = f_\kappa + f_{add}$. The polynomial $f_\kappa$ is of the certain form while $f_{add}$ is only restricted by the condition that the singularity of $f$ should be isolated. The polynomial $f_{add}$ is zero if and only if $f$ is invertible, and in the non-invertible case $f_{add}$ is arbitrarily complicated. In this paper we investigate all possible polynomials $f_{add}$ for a given non-invertible $f$. For a given $f_\kappa$ we introduce a specific small collection of monomials that build up $f_{add}$ such that the polynomial $f = f_\kappa + f_{add}$ defines an isolated quasihomogeneous singularity. If $(f,\ZZ/2\ZZ)$ is Landau-Ginzburg orbifold with such non-invertible polynomial $f$, we provide the quasihomogeneous polynomial $\bar{f}$ such that the orbifold equivalence $(f,\ZZ/2\ZZ) \sim (\bar{f}, \{\id\})$ holds. We also give an explicit isomorphism between the corresponding Frobenius algebras.
\end{abstract}
\maketitle
\setcounter{tocdepth}{1}
\section{Introduction}
Let $f$ be a quasihomogeneous polynomial defining an isolated singularity at the origin (in the remainder of this
article, we call such $f$
\textit{non-degenerate}) and let G be a group of its symmetries.  
Then the pair $(f,G)$ is called \emph{Landau-Ginzburg orbifold}. The studying of these objects was initiated by physicists in early 90s (cf. \cite{IV}, \cite{V}, \cite{Witt}) for the purposes of mathematical physics, and was followed by mathematicians. If $f$ is a representative of a rather special class of polynomials which are called \emph{invertible}, then there are many well-known results in singularity theory, algebraic geometry and mirror symmetry for Landau-Ginzburg orbifold associated with $f$ (cf. \cite{BT1}, \cite{BTW2}, \cite{EGZ}, \cite{BI}, \cite{FJJS}, \cite{KPA}). In particular, for diagonal groups $G$ there is a well-defined dual pair $(\Tilde{f},\Tilde{G})$ which is called \emph{BHK-dual} (see \cite{BHu}, \cite{BHe}), and mirror isomorphism between their Frobenius algebras (see \cite{Kr}). 

\subsection*{Non-invertible polynomials}
In this paper we will consider another class of 
polynomials, which are non-degenerate but non-invertible. Up to now there are quite few results in singularity theory and mirror symmetry related to such type of polynomials $f$ and corresponded Landau-Ginzburg orbifolds $(f,G)$ (cf. \cite{BT2}, \cite{ET2}). From \cite{HK} we have that any non-degenerate and non-invertible polynomial has a form $f = f_\kappa + f_{add}$, where $f_\kappa$ is constructed by a graph $\Gamma_f = \sqcup \Gamma_{f_i}$, where $\Gamma_{f_i}$ is a graph which we call \emph{loop with branches} graph (see Figure 1 in text), and $f_{add} = f - f_\kappa$. 

For example, consider two following polynomials $f_1 = x_1^{a_1}x_2 + x_2^{a_2}x_3 + x_3^{a_3}$ and
$f_2 = x_1^{a_1}x_2 + x_2^{a_2}x_3 + x_3^{a_3}x_2 + \varepsilon x_1^{b_1}x_3^{b_3}$. For non-degeneracy of polynomial $f_2$ we need $\varepsilon \in \CC^*$ and $b_1q_1+b_3q_3=1$, where $(q_1, q_2, q_3)$ is the set of weights of $f$ (see Remark 2). In this case $f_1$ is an example of invertible polynomial and $f_2$ is an example of non-invertible polynomial. The corresponding graphs look as follows:
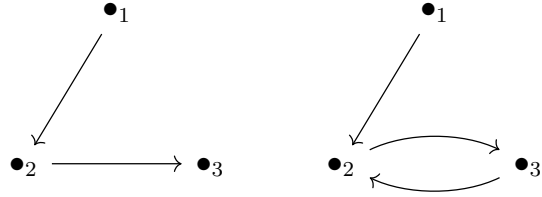
\[\begin{tikzcd}[column sep=small]
	& {\bullet_1} &&&& {\bullet_1} \\
	\\
	{\bullet_2} && {\bullet_3} && {\bullet_2} && {\bullet_3}
	\arrow[from=1-2, to=3-1]
	\arrow[from=1-6, to=3-5]
	\arrow[from=3-1, to=3-3]
	\arrow[curve={height=-12pt}, from=3-5, to=3-7]
	\arrow[curve={height=-12pt}, from=3-7, to=3-5]
\end{tikzcd}\]
 In particular, $(f_1)_\kappa = f_1$, $(f_{1})_{add} = 0$ and $(f_2)_\kappa = x_1^{a_1}x_2 + x_2^{a_2}x_3 + x_3^{a_3}x_2$, $(f_2)_{add} = \varepsilon x_1^{b_1}x_3^{b_3}$.

Suppose now that we start from the quasihomogeneous $f_\kappa$ corresponding to some graph of a quasihomogeneous singularity. If $f_\kappa$ is non-invertible, then it is degenerate (i.e. defines a non-isolated singularity at the origin). In our paper we introduce Theorem \ref{kappa-J} which provides admissible collection of monomials entering $f_{add}$ such that $f = f_\kappa + f_{add}$ is non-degenerate. Note that adding such $f_{add}$ given by the Theorem \ref{kappa-J} is not the unique method to construct a non-degenerate polynomial from $f_\kappa$. But our method is useful because it allows to control which monomials the polynomial $f_{add}$ consists of and on which variables it depends. We also introduce Theorem \ref{c} that provides the method of construction of $f_{add}$ which employs rather small number of monomials. In particular, $f_{add}$ depends only on variables with indices such that the arrow from corresponding vertex ends in vertex with index on the loop. In the second part we use such constructions to research corresponding Landau-Ginzburg orbifolds.
\subsection*{Orbifold equivalence}
The second part of this paper is devoted to orbifold equivalence $(f,G) \sim (\bar{f}, \{\id\})$ between Landau-Ginzburg orbifolds. It can be useful for the investigation of Landau-Ginzburg orbifolds with non-trivial group $G$. Orbifold equivalence could be roughly understood as the equivalence $\mathrm{MF}(\bar{f}) \cong \mathrm{MF}_G(f)$ of categories of $G$-equivariant matrix factorizations. In this case there is an isomorphism between Hochschild cohomology of the category $\mathrm{MF}_G(f)$ and an algebra $\Jac(\bar{f})$.

In our work we use the theorem introduced in \cite{BP}, \cite{Io} (Theorem \ref{equiv} in text), which provides the method to construct orbifold equivalence $(f,G) \sim (\bar{f}, \{\id\})$ by crepant resolution of $\CC^N/G$. In particular, we obtain the orbifold equivalence between pairs with polynomials that are non-invertible. We use the explicit form of polynomials with a \say{loop with branches} graph to construct a new polynomial corresponding to the same graph but with one additional edge (Theorem \ref{main} in text). We also use Shklyarov's \cite{Sh} techniques to calculate structure constants of $\ccHH^*(\mathrm{MF}_G(f))$. It helps us to find an explicit isomorphism between their Frobenius algebras (Proposition \ref{HHJ} in text). These results should help to research Frobenius structures and mirror symmetry for $(f, G)$ via the classical singularity structures for $\bar{f}$.

\subsection*{Acknowledgements}
I am grateful to my advisor Alexey Basalaev for his guidance and inspiration in the process of writing, useful
discussions at various stages of the preparation of this paper and comments on its draft. I am also grateful to the anonymous referee for their comments, which significantly enhanced the work. The work was supported by the Theoretical Physics and Mathematics Advancement Foundation «BASIS». 

\section{Preliminaries}
 Consider $\CC[\bx] := \mathbb{C}[x_1,x_2, \dots, x_N]$ a ring of polynomials with complex coefficients. We call a polynomial $f\in \CC[\bx]$ \emph{non-degenerate} if $f$ defines an isolated singularity at the origin, i.e. the system of equations $\{\frac{\partial f}{\partial x_1} = \frac{\partial f}{\partial x_2} = \dots = \frac{\partial f}{\partial x_N}\ = 0\}$ has the unique solution at $0$. Moreover, throughout this paper we assume that $f$ does not contain the summands of the type $x_ix_j$. 
 \begin{definition}
The polynomial $f\in \CC[\bx]$ is called \emph{quasihomogeneous} with a set of weights 
$(v_1,...,v_N,d)\in\ZZ_{\geq 0}^{N+1}$ if the equality 
\begin{align*}
f(\lambda^{v_1}x_1,\lambda^{v_2}x_2,...,\lambda^{v_N} x_N)=\lambda^d f(x_1,x_2,...x_N)
\end{align*}
is true for any $\lambda\in \CC^*$. 
\end{definition}
\begin{remark}
    In Sections 6 and 7 and we use the reduced system of weights $(q_1,q_2,...q_N)$, which can be obtained as $(q_1,q_2,...q_N) := (v_1/d, \dots, v_N/d)$. Moreover, $q_i \leq \frac{1}{2}$ (see \cite{HK}, Theorem 3.7).
\end{remark}
 Now let $f\in \CC[\bx]$ be a non-degenerate quasihomogeneous polynomial.

\begin{definition}
\emph{Jacobian algebra} of polynomial \textbf{$f$} is the quotient ring of the ring of polynomials $\CC[\bx]$ on the ideal generated by partial derivatives of $f$:
\begin{align*}
\mathrm{Jac}(f) := \mathbb{C}[x_1,x_2, \dots, x_N]/_{(\frac{\partial f}{\partial x_1}, \frac{\partial f}{\partial x_2}, \dots, \frac{\partial f}{\partial x_N})}
\end{align*}
\end{definition}
The dimension of a Jacobian algebra is called \emph{Milnor number} and is denoted by $\mu_f$, and we have $\mu_f < \infty$ if and only if $f$ is non-degenerate (see \cite{AGV}).


\begin{definition}
A non-degenerate quasihomogeneous polynomial $f\in\mathbb{C}[\mathbf{x}]$ is called \emph{invertible} if the following conditions are satisfied:
\begin{itemize}
\item 
The number of variables is the same as the number of monomials in $f$: 
\begin{align*}
f(x_1,\dots ,x_N)=\sum_{i=1}^N c_i\prod_{j=1}^n x_j^{E_{ij}}
\end{align*}
For some $c_i\in\mathbb{C}^\ast$ and integer non-negative 
$E_{ij}$ for $i,j=1,\dots, N$.
\item
The matrix $(E_{ij})$ is invertible over $\mathbb{Q}$.
\end{itemize}
\end{definition}
If non-degenerate quasihomogeneous $f$ does not satisfy the conditions above, the corresponding polynomial is called non-invertible, and in this work we are interested in this case.
\begin{definition}
The group of \emph{maximal diagonal symmetries} of $f$ is the group:
\begin{align*}
    G_f = \{ (g_1, g_2, \dots, g_N) \in (\mathbb{C}^*)^N \:|\: f(g_1x_1, g_2x_2, \dots, g_Nx_N ) = \\ f(x_1, x_2, \dots, x_N) \}
\end{align*}
\end{definition}
Any subgroup $G\subseteq G_f$ is called a group of \emph{diagonal symmetries} (or just a group of \emph{symmetries}). With each element $g\in G$ we can associate an algebra $\mathrm{Jac}(f^g)$ of the polynomial $f^g = f|_{Fix(g)}$ where $Fix(g) = \{x\in\mathbb{C}^N\:|\:g\cdot x= x \}$. In addition, if $Fix(g) = \{0\}$, we put $f^g := 1$.

\begin{proposition}[\cite{ET1}, Prop. 5]
    If the polynomial $f$ is non-degenerate, then for every $g\in G$ with non-trivial fixed locus, the polynomial $f^g$ is also non-degenerate.
\end{proposition}

\section{Polynomials and graphs}
\subsection{Combinatorial data}
Now we introduce following \cite{HK} the conditions on the weight system of quasihomogeneous $f$ such that it will be non-degenerate. Fix $N\in\NN$ and denote $I:=\{1,...,N\}$ and $e_i$ is the standard basis in the lattice $\ZZ_{\geq 0}^N$. For a subset $J\subseteq I$ and a system of weights $(v_1,...,v_N,d)\in\ZZ_{\geq 0}^{N+1}$ with $v_i<d$ and
$k\in\ZZ_{\geq 0}$ denote
\begin{eqnarray*}
\ZZ_{\geq 0}^J&:=& \{\alpha\in\ZZ_{\geq 0}^N\ |\ \alpha_i=0\textup{ for }i\notin J\},\\
(\ZZ^N_{\geq 0})_k&:=& \{\alpha\in \ZZ_{\geq 0}^N\ |\ \sum_i\alpha_i\cdot v_i=k\},\\
(\ZZ_{\geq 0}^J)_k&:=& \ZZ_{\geq 0}^J\cap (\ZZ_{\geq 0}^N)_k.
\end{eqnarray*}

\begin{lemma}[\cite{HK}, Lemma 2.1.]\label{t2.1}
Let us fix the system of weights $(v_1,...,v_N,d)\in\ZZ_{\geq 0}^{N+1}$ with $v_i<d$ and a subset
$R\subseteq (\ZZ_{\ge 0}^N)_d$. For any $k\in I$, we define the set
$$R_k:= \{\alpha\in (\ZZ_{\ge 0}^N)_{d-v_k}\ |\ \alpha+e_k\in R\}.$$
The following conditions are pairwise equivalent:
\begin{list}{}{}
\item[(C1):] 
\quad $\forall\ J\subset I\text{ such that } J\neq \varnothing$\\
\hspace*{2cm} $a)$ $\exists\ \alpha\in R\cap \ZZ_{\geq 0}^J$\\
\hspace*{1.6cm} or $b)$ $\exists\ K\subset I\backslash J\text{ such that }|K|=|J|$\\
\hspace*{5cm}$\text{ and }\forall\ k\in K\ \exists\ \alpha\in R_k\cap \ZZ_{\geq 0}^J.$\\
\item[(C1)':] 
\quad The same as (C1), but only for $J$ such that $|J|\leq \frac{N+1}{2}$.\\
\item[(C2):] 
\quad $\forall\ J\subset I\text{ such that }J\neq \varnothing$\\
\hspace*{2cm}$\exists\ K\subset I\text{ such that }|K|=|J|$\\
\hspace*{4cm}\text{ and }$\forall\ k\in K\ \exists\ \alpha\in R_k\cap \ZZ_{\geq 0}^J.$\\
\item[(C2)':] 
\quad The same as (C2), but only for $J$ such that $|J|\leq \frac{N+1}{2}$.
\end{list}
\end{lemma}
This lemma allows to formulate the criteria of non-degeneracy of a polynomial $f$ in a combinatorial way.

\begin{theorem}[\cite{HK}, \cite{Sa}, \cite{OPSh},\cite{KS}]{\label{NDcy}}
Let $(v_1,...,v_n,d)\in\NN^{N+1}$ be a system of weights with $v_i<d$ and for any $f\in \CC[\bx]$ define $supp(f) = \{ \alpha \in \ZZ_{\geq 0}^{N} \;|\; x_1^{\alpha_1}x_2^{\alpha_2}\dots x_N^{\alpha_N} \; is\; a\; term\; of\; f\} $.

(a) Let $f\in\CC[\bx]$ be a quasihomogeneous polynomial.
Then from the following condition
\begin{list}{}{}
\item[(IS1):] 
\quad $f$ is non-degenerate,
\end{list}

we have that the set $R := supp(f) \subset (\ZZ_{\geq 0}^N)_d$ satisfies the conditions (C1)-(C2)'.

\medskip
(b) Let $R$ be a subset $(\ZZ_{\geq 0}^N)_d$. Then, the following conditions are equivalent: 
\begin{list}{}{}
\item[(IS2):] \quad There is a a quasihomogeneous polynomial $f$ such that $ supp(f)\subseteq R$
and $f$ is non-degenerate.
\item[(IS2)':] \quad \quad A generic quasihomogeneous polynomial $f$ such that $supp(f) \subseteq R$ is non-degenerate.
\item[(C1) to (C2)':] \quad $R$ satisfies the conditions (C1)-(C2)'.
\end{list}
\end{theorem}
\subsection{Graph description of non-degeneracy}\label{comb}
We call a map $\kappa\colon I \to I$ a \emph{choice} if it satisfies the following condition: For every $ j \in I$ the sets $J = \{j\}$ and $K = \{\kappa(j)\}$ satisfy $(C2)$ for some fixed $R$. If we consider any quasihomogeneous polynomial $f$ and put $R=supp(f)$, then $\kappa$ is a choice if $f$ contains as a summand $b_j\cdot x_j^{a_j}\cdot x_{\kappa(j)}$, where $b_j\in \CC^*, a_j \in \NN$ and $a_j \geq 2$. Following this claim, polynomial $f$ has a form $f = f_\kappa + f_{add}$, where $f_\kappa$ is a part determined by the choice $\kappa$ and $f_{add} : = f - f_\kappa$. 
\begin{proposition}[\cite{HK}, Lemma 3.5.]
    Polynomial $f$ is invertible if and only if $f = f_\kappa$ for some $\kappa$. In particular, $f_{add} = 0$.
\end{proposition}
In this paper we work with non-invertible polynomials (i.e. $f_{add} \neq 0$) and to do this we use the graph method following \cite{HK}. We construct by a map $\kappa: I\to I$ the graph $\Gamma_\kappa$ with vertices labeled by the set $I$ and there is an oriented edge pointing from the $j$-th vertex to the $i$-th vertex if and only if $i = \kappa(j)$. We assume that if $\kappa(j) = j$, we do not draw the edge, therefore we obtain a graph without self-loops. We will call \emph{type} the conjugacy class $\kappa$ with respect to the natural action of the symmetric group on the set of indices. In this case the oriented graph without numbering of vertices defines the type.

\begin{theorem}[\cite{AGV}, \S 13.2]\label{AGV}
Let $f(x_1,x_2,x_3)$ be a non-degenerate polynomial of three variables. Then, the graphs corresponding to all possible choices for $f(x_1,x_2,x_3)$ are exhausted by the following graphs (types) up to permutation of variables:

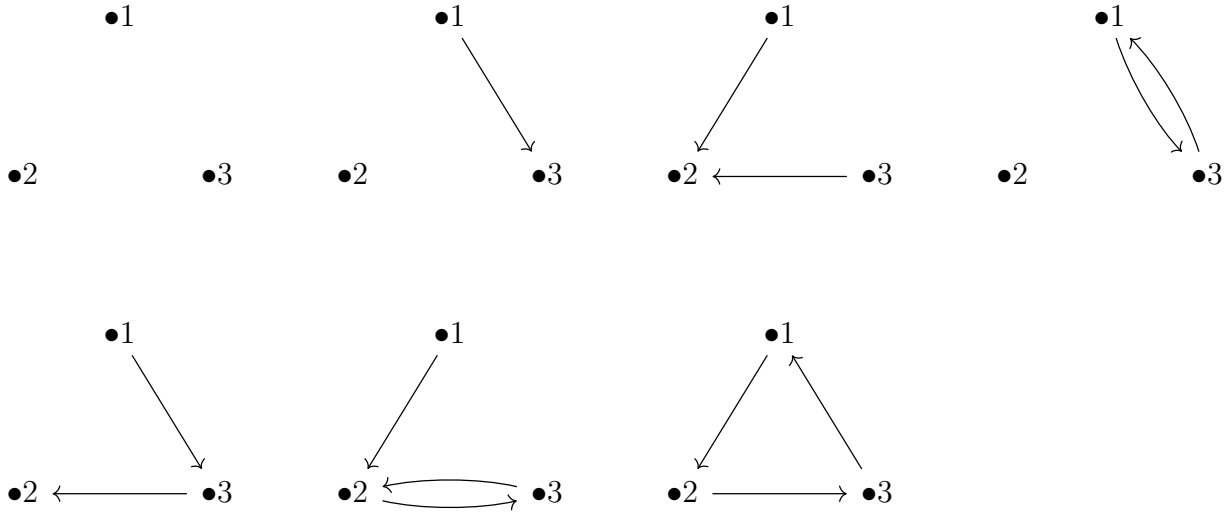
\[\begin{tikzcd}[column sep=small]
	& \bullet{1} &&&& \bullet{1} &&&& {\bullet}{1} &&&& \bullet{1} \\
	\\
	\bullet{2} && \bullet{3} && \bullet{2} && \bullet{3} && \bullet{2} && {\bullet}{3} && \bullet{2} && \bullet{3} \\
	\\
	& \bullet{1} &&&& {\bullet}{1} &&&& \bullet{1} \\
	\\
	\bullet{2} && \bullet{3} && \bullet{2} && {\bullet}{3} && \bullet{2} && \bullet{3}
	\arrow[from=1-6, to=3-7]
	\arrow[from=3-11, to=3-9]
	\arrow[from=1-10, to=3-9]
	\arrow[curve={height=6pt}, from=1-14, to=3-15]
	\arrow[curve={height=6pt}, from=3-15, to=1-14]
	\arrow[from=5-2, to=7-3]
	\arrow[from=5-6, to=7-5]
	\arrow[curve={height=6pt}, from=7-5, to=7-7]
	\arrow[curve={height=6pt}, from=7-7, to=7-5]
	\arrow[from=7-3, to=7-1]
	\arrow[from=5-10, to=7-9]
	\arrow[from=7-9, to=7-11]
	\arrow[from=7-11, to=5-10]
	\arrow[draw=none, from=3-1, to=3-3]
	\arrow[draw=none, from=3-5, to=3-7]
	\arrow[draw=none, from=3-13, to=3-15]
	\arrow[shift right=3, draw=none, from=7-5, to=7-7]
\end{tikzcd}
\]
\end{theorem}

In particular, polynomials corresponding to the graphs in Theorem \ref{AGV} have the following form:
\begin{itemize}
    \item $f_1 = x_1^{a_1}+x_2^{a_2}+x_3^{a_3}$  
    
    \item $f_2 = x_1^{a_1}x_3+ x_2^{a_2} + x_3^{a_3}$  
    
    \item $f_3 = x_1^{a_1}x_2+x_3^{a_3}x_2 + x_2^{a_2} + \varepsilon_{1,3}x_1^{b_1}x_3^{b_3}$, where $(a_2-1)\;|$ lcm$(a_1,a_3)$ and $\varepsilon_{1,3} \in \CC^*$
    
    \item $f_4 = x_2^{a_2} + x_1^{a_1}x_3 + x_3^{a_3}x_1$  
    
    \item $f_5 = x_1^{a_1}x_3 + x_3^{a_3}x_2  + x_2^{a_2}$ 
    
    \item $f_6 = x_1^{a_1}x_2 + x_2^{a_2}x_3 + x_3^{a_3}x_2 + \varepsilon_{1,3}x_1^{b_1}x_3^{b_3}$, where $(a_2-1)\cdot \text{gcd}(a_1,a_3)\; | \;(a_1-1)$ and $\varepsilon_{1,3} \in \CC^*$
    
    \item $f_7 = x_1^{a_1}x_2 + x_2^{a_2}x_3 + x_3^{a_3}x_1$ 
    
\end{itemize}

The conditions for the polynomials $f_3$ and $f_6$ described above are obtained if we explicitly write down the equation for the weights $(v_1,v_2,v_3)$ (see \cite{AGV}). In general there is the following statement, which describes the graphs that can be obtained:

\begin{proposition}[\cite{HK}, Lemma 3.1]{\label{arb}}
Exactly those graphs occur as graphs of maps $\kappa:I \to I$ whose components either are globally oriented trees or consist of one globally oriented cycle and finitely many globally oriented trees whose roots are on the cycle.
\end{proposition}

In other words, any connected component of $\Gamma_\kappa$ has the form of a loop with branches (see Fig. 1) including the case when the \say{loop} has length 1. It follows that any polynomial $f_\kappa$ has the following decomposition:
\begin{align*}
    f_\kappa = f_{inv} + f_1 + f_2 + \dots + f_p
\end{align*}
where $f_{inv}$ is an invertible polynomial, and $f_i$ is a polynomial corresponding to the $i$-th connected component of the graph $\kappa$.
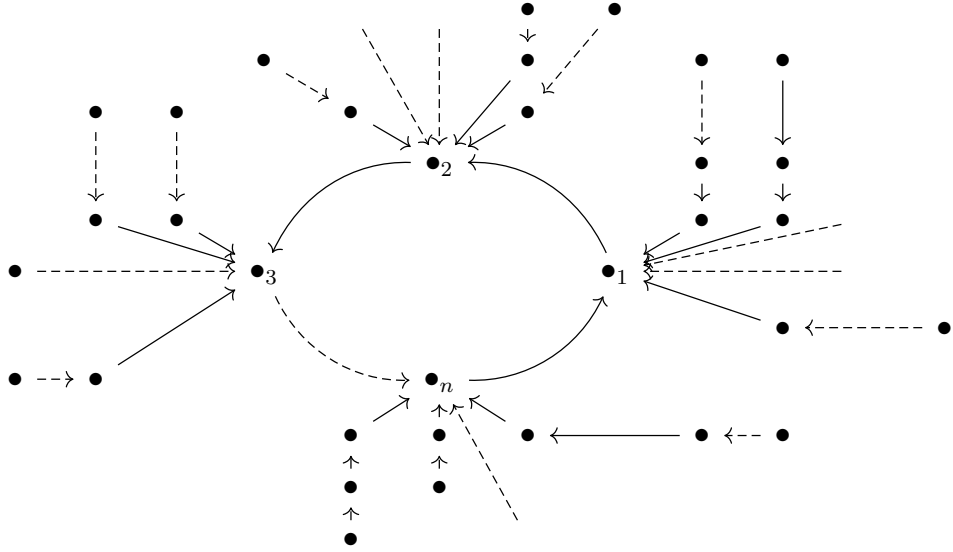
\begin{figure}[H]
\[\begin{tikzcd}[column sep=small,row sep=tiny]
	&&&& \textcolor{white}{\bullet} & \textcolor{white}{\bullet} & \bullet & \bullet \\
	&&& \bullet &&& \bullet && \bullet & \bullet \\
	& \bullet & \bullet && \bullet && \bullet \\
	&&&&& {\bullet_2} &&& \bullet & \bullet \\
	& \bullet & \bullet &&&&&& \bullet & \bullet & \textcolor{white}{\bullet} \\
	\bullet & \textcolor{white}{\bullet} && {\bullet_3} &&&& {\bullet_1} && \textcolor{white}{\bullet} & \textcolor{white}{\bullet} \\
	&&&&&&&&& \bullet && \bullet \\
	\bullet & \bullet &&&& {\bullet_n} \\
	&&&& \bullet & \bullet & \bullet && \bullet & \bullet \\
	&&&& \bullet & \bullet \\
	&&&& \bullet && \textcolor{white}{\bullet}
	\arrow[curve={height=18pt}, from=4-6, to=6-4]
	\arrow[curve={height=18pt}, dashed, from=6-4, to=8-6]
	\arrow[curve={height=18pt}, from=8-6, to=6-8]
	\arrow[curve={height=18pt}, from=6-8, to=4-6]
	\arrow[from=5-9, to=6-8]
	\arrow[from=4-9, to=5-9]
	\arrow[dashed, from=2-9, to=4-9]
	\arrow[from=5-10, to=6-8]
	\arrow[from=4-10, to=5-10]
	\arrow[from=2-10, to=4-10]
	\arrow[from=7-10, to=6-8]
	\arrow[dashed, from=6-11, to=6-8]
	\arrow[dashed, from=5-11, to=6-8]
	\arrow[from=3-7, to=4-6]
	\arrow[dashed, from=1-8, to=3-7]
	\arrow[from=2-7, to=4-6]
	\arrow[dashed, from=1-7, to=2-7]
	\arrow[dashed, from=2-4, to=3-5]
	\arrow[from=3-5, to=4-6]
	\arrow[dashed, from=1-5, to=4-6]
	\arrow[dashed, from=1-6, to=4-6]
	\arrow[from=5-3, to=6-4]
	\arrow[dashed, from=3-3, to=5-3]
	\arrow[dashed, from=6-1, to=6-4]
	\arrow[from=5-2, to=6-4]
	\arrow[dashed, from=3-2, to=5-2]
	\arrow[from=8-2, to=6-4]
	\arrow[dashed, from=7-12, to=7-10]
	\arrow[dashed, from=8-1, to=8-2]
	\arrow[from=9-5, to=8-6]
	\arrow[from=10-5, to=9-5]
	\arrow[dashed, from=11-5, to=10-5]
	\arrow[from=9-6, to=8-6]
	\arrow[from=9-7, to=8-6]
	\arrow[dashed, from=10-6, to=9-6]
	\arrow[dashed, from=11-7, to=8-6]
	\arrow[from=9-9, to=9-7]
	\arrow[dashed, from=9-10, to=9-9]
\end{tikzcd}\]
    \caption{Loop with branches graph}
    \label{fig:enter-label}
\end{figure}

Since we are interested in the non-invertible case, in what follows we denote by $f_\kappa$ a polynomial with only one connected component arising from the loop with branches graph and a set of powers $(a_1,\dots, a_N)$ defining the weight system $(v_1,...,v_N,d)\in\ZZ_{\geq 0}^{N+1}$. 

Notice every variable corresponds to a node with only one outgoing
arrow. Hence we obtain the claim:
\begin{proposition}{\label{scaling}}
    Let $f_\kappa$ = $\sum_{j=1}^{N}b_jx_j^{a_j}x_{\kappa(j)}$. Then we can assume that all $b_j$ are equal to 1.
\end{proposition}

Let us also introduce the gluing a point operation on the graphs. Let $\kappa\colon I \to I$ be the map defining the loop with branches graph $\Gamma_\kappa$ and let $T_\kappa \subset I$ be the set of leaves of the graph $\Gamma_\kappa$, i.e. the set of vertices on
the end of each branch. Let $t \in T_\kappa$ be a leaf and $\{N+1\}$ is an isolated vertex. Define the operation $\circ_{\Gamma_\kappa}(N+1, t)$ which adds an arrow from the isolated vertex $N+1$ to the leaf $t$. It means that $\circ_{\Gamma_\kappa}(N+1, t) = \Gamma_{\bar{\kappa}}$, where $\bar{\kappa}$ is a new graph with glued vertex. We could also expand the action of $\circ_{\kappa}$ on $k$ isolated vertices and $k$ leaves. In this case we have $\circ_{\Gamma_\kappa}((N+1, \dots, N+k), (t_1, \dots, t_k))$ and connect $N+i$ with $t_i$. This construction will be useful in Section 6 and 7 to describe the new polynomial $\bar{f}$ with graph obtaining by this action and show the construction of orbifold equivalence.
\begin{example}
Let us consider the following graph $\Gamma_{\kappa}$ and an isolated vertex \{10\}:
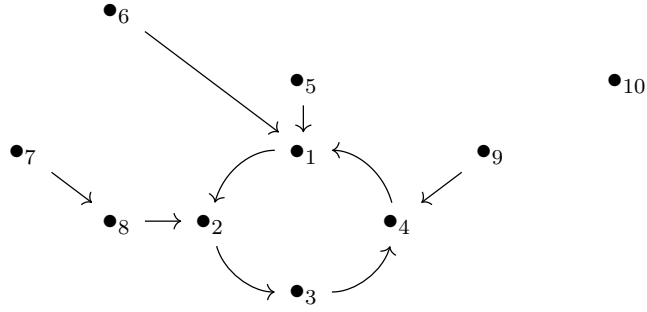
\[\begin{tikzcd}[sep=small]
	& {\bullet_6} \\
	&&& {\bullet_5} &&&& {\bullet_{10}} \\
	{\bullet_7} &&& {\bullet_1} && {\bullet_9} \\
	& {\bullet_8} & {\bullet_2} && {\bullet_4} \\
	&&& {\bullet_3}
	\arrow[curve={height=12pt}, from=3-4, to=4-3]
	\arrow[curve={height=12pt}, from=4-3, to=5-4]
	\arrow[curve={height=12pt}, from=5-4, to=4-5]
	\arrow[curve={height=12pt}, from=4-5, to=3-4]
	\arrow[from=1-2, to=3-4]
	\arrow[from=2-4, to=3-4]
	\arrow[from=4-2, to=4-3]
	\arrow[from=3-1, to=4-2]
	\arrow[from=3-6, to=4-5]
\end{tikzcd}\]
Then $\circ_{\Gamma_\kappa}(\{10\},\{9\} ) = \Gamma_{\kappa_2}$ and $\circ_{\Gamma_\kappa}(\{10\}, \{5\}) = \Gamma_{\kappa_3}$, where $\Gamma_{\kappa_2}$ and $\Gamma_{\kappa_3}$ are the following graphs:
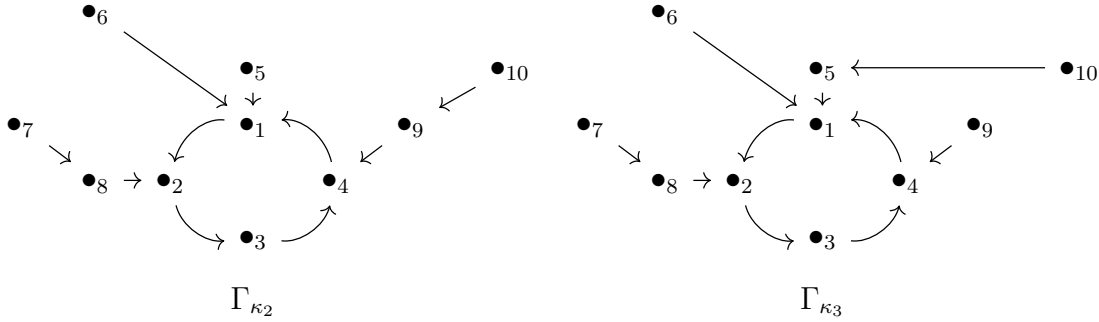
\[\begin{tikzcd}[sep=tiny]
	& {\bullet_6} &&&&&&&& {\bullet_6} \\
	&&& {\bullet_5} &&&& {\bullet_{10}} &&&& {\bullet_5} &&&& {\bullet_{10}} \\
	{\bullet_7} &&& {\bullet_1} && {\bullet_9} &&& {\bullet_7} &&& {\bullet_1} && {\bullet_9} \\
	& {\bullet_8} & {\bullet_2} && {\bullet_4} &&&&& {\bullet_8} & {\bullet_2} && {\bullet_4} \\
	&&& {\bullet_3} &&&&&&&& {\bullet_3} \\
	&&& {\Gamma_{\kappa_2}} &&&&&&&& {\Gamma_{\kappa_3}}
	\arrow[curve={height=12pt}, from=3-4, to=4-3]
	\arrow[curve={height=12pt}, from=4-3, to=5-4]
	\arrow[curve={height=12pt}, from=5-4, to=4-5]
	\arrow[curve={height=12pt}, from=4-5, to=3-4]
	\arrow[from=1-2, to=3-4]
	\arrow[from=2-4, to=3-4]
	\arrow[from=4-2, to=4-3]
	\arrow[from=3-1, to=4-2]
	\arrow[from=3-6, to=4-5]
	\arrow[from=2-8, to=3-6]
	\arrow[from=1-10, to=3-12]
	\arrow[from=2-12, to=3-12]
	\arrow[curve={height=12pt}, from=3-12, to=4-11]
	\arrow[curve={height=12pt}, from=4-11, to=5-12]
	\arrow[curve={height=12pt}, from=5-12, to=4-13]
	\arrow[curve={height=12pt}, from=4-13, to=3-12]
	\arrow[from=4-10, to=4-11]
	\arrow[from=3-9, to=4-10]
	\arrow[from=3-14, to=4-13]
	\arrow[from=2-16, to=2-12]
\end{tikzcd}\]
\end{example}

\section{Non-invertible quasihomogeneous singularities}
Now we start with a map $\kappa$ and the systems of weights $(v_1,...,v_N,d)\in(\ZZ_{\ge 0}^{N+1})_d$, and construct by them quasihomogeneous polynomial $f_\kappa$ (may be degenerate) and consider $R=supp(f_\kappa)$. In this section our aim is to obtain the explicit non-degenerate polynomial $f$ such that $f = f_\kappa + f_{add}$ starting from the fixed $f_\kappa$.  
\begin{definition}
A (nonempty) subset of indexes $J \subset I = \{1,\dots, N\}$ is called \emph{failing set} for a given $R$, if it does not satisfy the condition $(C1)$ (or any other equivalent condition), for instance:

\begin{list}{}{}
\item[(NC1):]  
\quad $R\cap \ZZ_{\geq 0}^J = \varnothing$\\
\hspace*{1.6cm}and $\forall\ K\subset I\backslash J\text{ such that }|K|=|J|$\\
\hspace*{5cm}$\text{ it is true that }\exists\ k\in K\ R_k\cap \ZZ_{\geq 0}^J = \varnothing$
\end{list}
\end{definition}

It means also that $R$ satisfies the conditions of Lemma $\ref{t2.1}$ if there is no failing set for this $R$. Note that if there are failing sets for $R$ then $f_\kappa$ is degenerate by Theorem \ref{NDcy}. We want to describe the connection between failing sets of such $R$ and the map $\kappa$.
\begin{remark}
    Note that for a given graph $\Gamma_\kappa$ the set of powers $(a_1,\dots,a_N)$ defines the system of weights $(v_1,...,v_N,d)\in(\ZZ_{\ge 0}^N)_d$. We will use it in what follows, since sometimes it is easier to work with powers.
\end{remark}
\begin{proposition}
    Let $J$ be a failing set for $R = supp(f_\kappa)$. Then $\forall j \in \{1,2,\dots,N\}$, $\{j, \kappa(j)\} \nsubseteq J$. Moreover, the statement is also true for fixed points of $\kappa$: $\kappa(j) = j$.
\end{proposition}
\begin{proof}
    In this case $a_je_j+e_{\kappa(j)} \in R\cap \ZZ_{\geq 0}^J \neq \varnothing$ and by $(C1a)$ $J$ is not failing.
\end{proof}
Consider the example and find failing sets for some degenerate quasihomogeneous polynomial $f_\kappa$.

\begin{example}{\label{ea_f}}
    Let start with the map $\kappa:\{1,2,3,4\}\to \{1,2,3,4\}$ which corresponds to the graph
    
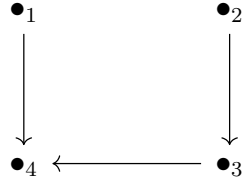
\[\begin{tikzcd}
	{\bullet_1} && {\bullet_2} \\
	\\
	{\bullet_4} && {\bullet_3}
	\arrow[from=3-3, to=3-1]
	\arrow[from=1-1, to=3-1]
	\arrow[from=1-3, to=3-3]
\end{tikzcd}\]
and an arbitrary set of powers $(a_1,a_2,a_3,a_4)$ ($a_i \geq 2$). Then we have a polynomial $f_\kappa$ = $x_4^{a_4+1} + x_1^{a_1}x_4+x_3^{a_3}x_4+x_2^{a_2}x_3$ which is degenerate and consider $R = supp(f_\kappa)$. Let us write the matrix such that $R$ coincides with the set of it's rows:
\[
  \left[ {\begin{array}{cccc}
    0 & 0 & 0 & a_4+1 \\
    0 & 0 & a_3 & 1 \\
    0 & a_2 & 1 & 0 \\
     a_1 & 0 & 0 & 1 \\
  \end{array} } \right]
\]
We would like to find all failing sets for this $R$. Since $\kappa$ is a choice, then all $J \subset I$ with $|J|=1$ are not  failing sets. Therefore we consider the case $|J|=2$ (since by (C1)' we could consider only the sets $J$ with $|J| \leq \frac{5}{2} < 3) $. Notice that  $\forall j \in \{1,2,3,4\}$, $J \neq \{j, \kappa(j)\}$ by the proposition above.  We could also notice that if $4 \in J$, then $J$ is also not a failing set by $(C1a)$, since $(a_4+1)e_4 \in R\cap \ZZ_{\geq 0}^J \neq \varnothing$. It follows that the only possible failing set is $J = \{1,3\}$ (since for $J = \{1,2\}$ the set $K = \{3,4\}$ satisfies the condition (C1)). Actually $R\cap \ZZ_{\geq 0}^J = \varnothing$, and the only possible option remains $K = \{2,4\}$, for which the condition $R_1\cap \ZZ_{\geq 0}^J = R_3\cap \ZZ_{\geq 0}^J = \varnothing$ follows. Consequently, $J = \{1,3\}$ is a failing set for $R$ presented above.
\end{example}

We have to introduce the following definition to describe the construction of $f_{add}$:

\begin{definition}
    Let $(v_1,...,v_N,d)\in\ZZ_{\geq 0}^{N+1}$ be a system of weights, $R\subset (\ZZ_{\geq 0}^N)_d$ be a subset of a lattice and $F_R$ $\subset 2^I$ be a set of all failing sets for $R$. A collection of sets $A_R := \{J_1, \dots, J_l$\} with $J_i \subset I$ is called \emph{admissible} for $R$ if the following conditions hold:
    \begin{enumerate}
    \item $J_k$ is a failing set for $R$ for $1 \leq k \leq l$, i.e. $A_R \subset F_R$
    \item For any $J \in F_R$ there is $J_k \in A_R$ such that $J \backslash J_k$ is not failing.
\end{enumerate}
\end{definition}

Now we are ready to introduce the following theorem: 

\begin{theorem}\label{kappa-J}
    Let $A_R$ be an admissible collection of $R = supp(f_\kappa)$ and assume there is a set of multipowers $\{b_{J_K} \in (\ZZ_{\ge 0}^{J_K})^d | J_K \in R \;\text{and}\; b_s>0\; \text{for}\; s\in J_K\}$.
    
    Then, there is a set $\{\epsilon_{J_k} \in \CC^*\}$ such that the polynomial $f = f_\kappa + f_{add}$ is non-degenerate with
\[f_{add} = \sum_{J_k \in A_R} \varepsilon_{J_k}x_{k_1}^{b_{k_1}}x_{k_2}^{b_{k_2}} \dots x_{k_l}^{b_{k_l}} \]
\end{theorem}
Before proving of this theorem, let us introduce the example which illustrates the construction of $f_{add}$.
\begin{example}
    Let us turn back to Example \ref{ea_f} and put $(a_1,a_2,a_3,a_4) = (6,9,3,7)$ that define the system of weights $(v_1,v_2,v_3,v_4,d) = (9, 5, 18, 9, 63)$. The only failing set for $f_\kappa =x_4^{7} + x_1^{6}x_4+x_3^{3}x_4+x_2^{9}x_3$ is $J=\{1,3\}$. So the admissible collection is unique, and we have $A_R = \{J\}$. Note that $b_1 = 3$ and $b_3 = 2$ satisfy the condition $b_1v_1 + b_3v_3 = 63$. It follows from the theorem above that:
\[
f = f_\kappa + f_{add} = x_4^{7} + x_1^{6}x_4+x_3^{3}x_4+x_2^{9}x_3 + x_1^{3}x_3^{2}
\]
is non-degenerate what can be checked by direct computations (here we have $\varepsilon_J = 1$).
\end{example}

To prove the theorem we need to introduce the following lemma.

\begin{lemma}\label{R-f}
Let $A_R = \{J_1,\dots, J_l\}$ be admissible for $R$. Then for the set
 \begin{align}
R' := R \cup \ZZ_{\geq 0}^{J_1} \cup \ZZ_{\geq 0}^{J_2} \cup \dots \cup \ZZ_{\geq 0}^{J_l} \subset (\ZZ_{\geq 0}^N)_d
\end{align}
there is no any failing set.
 \end{lemma}

 \begin{proof} Suppose that $J \subset I$ is a failing set for $R'$ and is not a failing set for $R$. Since $R \subset R'$, then $(C2a)$ for $R$ can not be satisfied for $J$ (otherwise $J$ is not failing for $R'$). Therefore, $(C1b)$ for $R$ should be satisfied and $\exists$ $K \subset I$ such that $|K| = |J|$ and $\forall\; k \in K$ $\exists\ \alpha \in R_k\cap \ZZ_{\geq 0}^J$. But since $R_k \subset R'_k$ we also have $\alpha \in R'_k\cap \ZZ_{\geq 0}^J$. Consequently, $(C1b)$ is satisfied for $R'$ and $J$ is not a failing set for $R'$.

 Suppose now that $J$ is a failing set for $R$. By the definition of admissible collection $\exists\; J_k \in A_R$ such that $J = J_{k} \sqcup J_{k}^c$, where $J_{k}^c$ is a complementary set. It means that $\ZZ_{\geq 0}^{J_k} \neq \varnothing$ and in this case we have \[\ZZ_{\geq 0}^J \cap R' = \ZZ_{\geq 0}^J \cap (R \cup \dots \cup \ZZ_{\geq 0}^{J_k}) \supset \ZZ_{\geq 0}^{J_k} \neq \varnothing\]
 and from $(C1a)$ it follows that $J$ is not a failing set for $R'$ that completes the proof.
 \end{proof}

Now we prove Theorem \ref{kappa-J}.
\begin{proof}
We apply the lemma above to prove the theorem. Note that $R' = supp(f_\kappa) \cup supp(f_{add})$ for arbitrary non-zero coefficients in $f_{add}$. Due to the condition $\sum_{s \in J_k} b_{s}v_{s} = d$ and by $(IS 2)'$ we have the generic non-degenerate polynomial $f$ with the support equal to $R'$, i.e. $f$ is a linear combination of monomials from $f_\kappa$ and $f_{add}$ with some non-zero coefficients. By Proposition \ref{scaling} we assume that coefficients in $f_\kappa$ are equal to one and by these linear transformations we obtain the necessary set $\{\epsilon_{J_k} \in \CC^*\}$.
\end{proof}


\section{Admissible collections for the loop with branches}
In this Section we start with a polynomial $f_\kappa$ and present the method how to construct an admissible collection $A_R$ for $R=supp(f_\kappa)$. This gives us a recipe for taking any quasihomogeneous polynomial $f_\kappa$ from a loop with branches graph, and turning it into a non-degenerate quasihomogeneous polynomial.
Suppose we have a loop with branches graph on $n$ vertices (i.e. the length of the loop is equal to $n$).
For any $m \in \{1,\dots,n\}$ consider the set $S_m \subset \{1,\dots, N\}$ consisting of all vertices from which we have an arrow ending in $m$. Then we also define $A^m := \{ J_l \subset S_m \;|\; |J_l| \geq 2 \}$ to be the set of all subsets of $S_m$ with at least two elements. 
\begin{figure}[H]
    \centering
    \[\begin{tikzcd}
	&&&& \bullet & \textcolor{white}{\bullet} \\
	&&& {*} & {*} & \bullet & \textcolor{white}{\bullet} \\
	& \textcolor{white}{\bullet} & \bullet && {*} & \bullet & \textcolor{white}{\bullet} \\
	\textcolor{white}{\bullet} &&& {\bullet_m} && {*} & \bullet & \textcolor{white}{\bullet} \\
	& \bullet & {*} && {*} & \bullet & \textcolor{white}{\bullet}
	\arrow[from=2-4, to=4-4]
	\arrow[from=1-5, to=2-4]
	\arrow[from=2-5, to=4-4]
	\arrow[from=3-5, to=4-4]
	\arrow[from=2-6, to=2-5]
	\arrow[from=3-6, to=3-5]
	\arrow[from=4-6, to=4-4]
	\arrow[from=5-5, to=4-4]
	\arrow[from=5-6, to=5-5]
	\arrow[curve={height=6pt}, from=5-3, to=4-4]
	\arrow[curve={height=6pt}, from=4-4, to=3-3]
	\arrow[from=4-7, to=4-6]
	\arrow[dashed, from=5-7, to=5-6]
	\arrow[dashed, from=3-7, to=3-6]
	\arrow[dashed, from=2-7, to=2-6]
	\arrow[dashed, from=1-6, to=1-5]
	\arrow[dashed, from=4-8, to=4-7]
	\arrow[curve={height=6pt}, dashed, from=3-3, to=3-2]
	\arrow[curve={height=6pt}, from=5-2, to=5-3]
	\arrow[curve={height=6pt}, dashed, from=4-1, to=5-2]
\end{tikzcd}\]
    \caption{The elements of $S_m$ are marked as stars}
\end{figure}
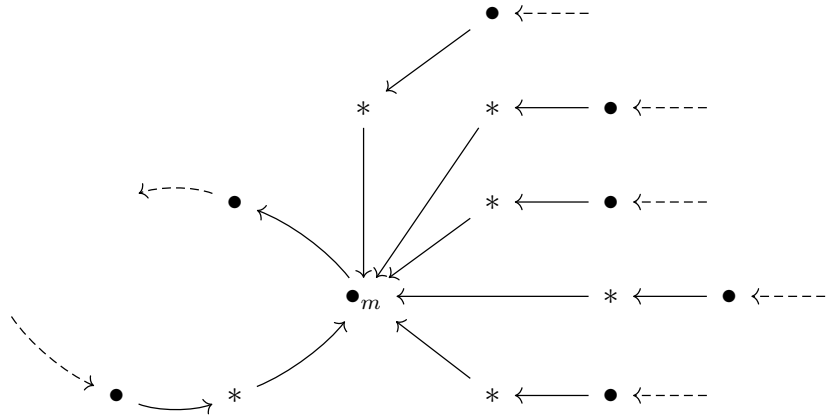
Now we are ready to formulate the following Theorem:
\begin{theorem}\label{c}
    The set $A_R = \cup_{m=1}^n A^m$ is an admissible collection for $R = supp(f_\kappa)$. 
\end{theorem}

\begin{proof}
Firstly, we are going to show that any $J_l \in A_R$ is a failing set for $R$. 
If it is not, then by (C1) there is the set $K \subset I$ such that $|K| = |J_l|$ and for all $k \in K$ there is $\alpha_k \in R_k\; \cap\; \ZZ_{\geq 0}^{J_l}$ (since $R \cap\; \ZZ_{\geq 0}^{J_l} = 0)$. It means, that for all $k\in K$ we have $\alpha_k + e_k \in R$ and $\alpha_k \in \ZZ_{\geq 0}^{J_l}$, but $R = supp(f_\kappa)$ consists of elements of $\ZZ_{\geq 0}^{N}$ with only two non-zero coordinates (and one of them is equal to one). We conclude that $\alpha_k$ should have two non-zero coordinates (with one of them equal to one) or only one non-zero coordinate. The first case is impossible since it means the arrow between some $j_1, j_2 \in J_l$, and the second case is also impossible since it means that $K = \{k\}$, where $\kappa(j_1) = \kappa(j_2) = \dots = k$ is the unique set satisfies the necessary condition (C1) and $|K| = 1 < |J_l|$ . From this we can conclude that any  $J_k \in A^m \subset A_R$ is a failing set for $R$ for any $m$.

Now we want to show that for any failing set $J$ there is $J_l \in A_R$ such that $J_l \subset J$. From above we conclude that any failing set for $R$ has a decomposition $J = U^1_J \sqcup U^2_J \sqcup \dots U^n_J \sqcup U_J^c$, where $U^m_J$ consists all indices from $I$ whose image under $\kappa$ is equal to $m \in \{1,\dots,n\}$ and $U_J^c$ is the complementary set. Since $J$ is failing there is $s$ such that $|U^s_J| \geq 2$. It means, that there is $J_l \in A_R$ such that $U^s_J = J_l$. Consequently,  $J_l \subset J$. We show that for any failing set $J$ there is a set $J_k \in A_R$ such that $J_l \subset J$ what gives us that the collection $A_R$ is admissible for $R$.
\end{proof}

\begin{corollary}
    Let $A_R$ be a collection as above. If $f_{add}$ is from Theorem \ref{kappa-J}, then $f = f_\kappa + f_{add}$ is non-degenerate.
\end{corollary}
\begin{proof}
    Follows by Theorem \ref{NDcy}.
\end{proof}
\begin{example}
  Consider the map $\kappa$ which defines the following graph:
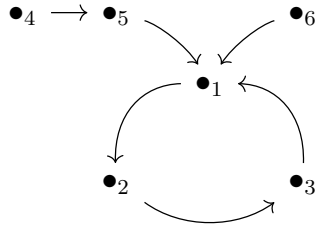
\[\begin{tikzcd}[sep = small]
	{\bullet_4} & {\bullet_5} && {\bullet_6} \\
	&& {\bullet_1} \\
	\\
	& {\bullet_2} && {\bullet_3}
	\arrow[curve={height=18pt}, from=2-3, to=4-2]
	\arrow[curve={height=18pt}, from=4-2, to=4-4]
	\arrow[curve={height=18pt}, from=4-4, to=2-3]
	\arrow[from=1-1, to=1-2]
	\arrow[curve={height=-6pt}, from=1-2, to=2-3]
	\arrow[curve={height=6pt}, from=1-4, to=2-3]
\end{tikzcd}\]
and the set of powers $(a_1,a_2,a_3,a_4,a_5,a_6) = (3,2,4,2,3,4)$ that defines the weight system $(v_1,v_2,v_3,v_4,v_5,v_6,d) = (1,2,1,1,2,1,5)$, i.e. the polynomial defined by $\kappa$ has the form
\begin{align*}
  f_\kappa = x_1^3x_2+x_2^2x_3+x_3^4x_1+x_5^2x_1+x_4^3x_5+x_6^4x_1  
\end{align*}
and $R = supp(f_\kappa)$. By theorem above, the admissible collection $A_R$ has the form $A_R = A^1 = \{\{3,5\}, \{3,6\}, \{5,6\}, \{3,5,6\}\}$. Thus the set of corresponding $\varepsilon_{J_k}$ exists and we take:
\[
f_{add} = \varepsilon_{3,5}x_3x_5^2 + \varepsilon_{3,6}x_3^2x_6^3+\varepsilon_{5,6}x_5^2x_6 + \varepsilon_{3,5,6}x_3^2x_5x_6\]
\[
f = f_\kappa + f_{add} = x_1^3x_2+x_2^2x_3+x_3^4x_1+x_5^2x_1+x_4^3x_5+x_6^4x_1 + \varepsilon_{3,5}x_3x_5^2 + \varepsilon_{3,6}x_3^2x_6^3+\varepsilon_{5,6}x_5^2x_6 + \varepsilon_{3,5,6}x_3^2x_5x_6
\]
where by Theorem \ref{kappa-J} the polynomial $f$ is non-degenerate what could be verified by exact calculation (exactly at that case we can put $\varepsilon_{J_k} = 1$ for all $J_l \in A_R$).
\end{example}

\section{Crepant resolutions and orbifold equivalence}\label{CR and OE}
In this Section we are going to consider Landau-Ginzburg orbifold $(f, G)$, where quasihomogeneous polynomial $f$ corresponds to loop with branches and 1 isolated vertex and the group $G$ is isomorphic to $\ZZ/2\ZZ$. 
We describe a new polynomial $\bar{f}$ which is obtained from $f$ using the gluing a point operation and introduce the orbifold equivalence $(\bar{f}, \{\id\}) \sim (f, G)$. 
\subsection{Background}

Let $f\colon \CC^N \to \CC$ be a holomorphic $G$-invariant function. Thus it could be defined on $\CC^N/G$. Let $\tau\colon \widehat{\CC^N/G} \to \CC^N/G$ be a crepant resolution of the
singularity. The map $f$ then lifts to a function on $\widehat{\CC^N/G}$ by taking composition with $\tau$ and we obtain a function $\hat{f}\colon \widehat{\CC^N/G}\to \CC$. Let $\widehat{\CC^N/G}$ be covered by some charts $U_1, \dots, U_s$ all isomorphic to $\CC^N$. We denote $\hat{f}_i$ as a restriction of $\hat{f}$ on the each chart 
\[\hat{f}_i = \hat{f}|_{U_i}\colon U_i \to \CC\]

If we consider a non-degenerate quasihomogeneous polynomial $f \in \CC[\textbf{x}]$ and a group of symmetries $G$, we can associate with the pair $(f,G)$ the category $\mathrm{MF}_G(f)$ (we refer to \cite{Orl1}, \cite{Orl2}, \cite{Io} for details). In the category language we formulate the definition of orbifold equivalence (see \cite{Io}) and have the following theorem:
\begin{theorem}\label{equiv}(\cite{BP}, \cite{Io})
    Let $f$ and $\hat{f}$ both have isolated singularities at the unique points $v \in \CC^N/G$ and $w \in \widehat{\CC^N/G}$ respectively. Then the pairs $(f,G)$ and $(\hat{f},\{\id\})$ are orbifold equivalent. Namely there is an equivalence of categories
$$\mathrm{MF}(\hat{f}) \cong\mathrm{MF}_G(f)$$

Moreover, suppose that $\hat{f}$ has a singularity at the origin and there is an affine chart $U_i$ such that $\hat{f}_i\colon \CC^N \cong U_i \to \CC$ has a singularity at $0$ and $\hat{f}_j\colon \CC^N \cong U_j\to \CC$ does not have any singularities $\forall j\neq i$ . Then:
$$\mathrm{MF}(\hat{f}) \cong \mathrm{MF}(\hat{f}_i)$$
\end{theorem}

It means that for orbifold equivalence between $(f, G)$ and $(\bar{f}, \{\id\})$ it is sufficient to find the suitable affine chart such that $\bar{f} = \hat{f}_m$ and $\hat{f}_m$ is the unique non-degenerate polynomial in the collection $\{\hat{f}_i\}$.

\subsection{Orbifold equivalence for the loop with branches}
Now we start with a polynomial $f_{\kappa_0}$ with one connected component, and assume that there is an index $t \in T_{\kappa_0}$ (i.e. the set of leaves) such that the corresponding power is even, i.e. there is a monomial $x_t^{2a_t}x_{\kappa(t)}$. Without loss of generality we put $t=1$ and $\kappa(1)=2$.

By Theorem \ref{c} we obtain $f_0 = f_{\kappa_0} + f_{add}$, where $f_{add}$ is given by the admissible collection for $R = supp(f_{\kappa_0})$. Then we consider the polynomial $f_\kappa = f_{\kappa_0} + x_{N+1}^2$ with the set of powers $(2a_1, \dots, a_N, 2)$ which is quasihomogeneous with the reduced set of weights $(q_1,\dots,q_N, 1/2)$. Note that the corresponding graph $\Gamma_\kappa$ is a disjoint union of $\Gamma_{\kappa_0}$ and 1 isolated vertex. Since $x_{N+1}^2$ is invertible, the admissible collection $A_R$ for $R = supp(f_\kappa)$ coincides with the admissible collection for $R = supp(f_{\kappa_0})$. It follows that $f = f_\kappa + f_{add}$ is non-degenerate by Theorem \ref{c} (with the same $f_{add}$ as for $f_0$).

Our object of research is Landau-Ginzburg orbifold ($f, G)$ with the symmetry group $G\cong \ZZ/2\ZZ \cong \langle g\rangle$ acting as follows:
\[g\cdot x_m = - x_m \;\text{if}\; m = 1;N+1\] 
\[g\cdot x_{m} =  x_{m} \; \text{else} \]
It means that $g$ acts non-trivially only on isolated vertex and on the leaf with even power.

For the group $G$ described above we have $\CC^{N+1}/G \cong \CC^{N-1} \times \{w^2 = uv\}\subset \CC^{N+1}$ by identifying$\{u = x_1^2, v = x_{N+1}^2, w = x_1x_{N+1}\}$ or $\{v = x_1^2, u = x_{N+1}^2, w = x_1x_{N+1}\}$ which gives us $2$ charts $U_1$ and $U_2$ covering $\CC^{N+1}/G$, where for $U_1$:
\[(x_{2}, \dots, x_{N}, y, z) \longrightarrow (x_{2}, \dots, x_{N}, y,  yz^2)\]
and for $U_2$:
\[(x_{2}, \dots, x_{N}, y, z) \longrightarrow (x_{2}, \dots, x_{N}, y^2z, z)\]

Now we are ready to formulate the following theorem:
\begin{theorem}\label{main}
    Let Landau-Ginzburg orbifold $(f,G)$ be as above. Then there is a non-degenerate quasihomogeneous polynomial $\bar{f}$ with the following properties:
    \begin{enumerate}
        \item $\bar{f}$ has a reduced system of weights $(2q_1, q_2, \dots, q_N, 1/2 - q_1)$
        \item  $ \bar{f} = f_{\bar{\kappa}} + \bar{f}_{add}$ with the graph $\Gamma_{\bar{\kappa}} = \circ_{\Gamma_\kappa}(N+1, 1)$ and $\bar{f}_{add}= f_{add}(t_1, \dots, t_N)$ where we put $t_1^2 = x_1$ and $t_m=x_m$ for $m \neq 1$;
    \end{enumerate}
such that $(f, G)$ and $(\bar{f}, \{\id\})$ are orbifold equivalent. In particular there is an equivalence of categories
$$\mathrm{MF}(\bar{f}) \cong\mathrm{MF}_G(f).$$
\end{theorem}
\begin{example}
    Let us consider the example starting with a polynomial $f_{\kappa_0} = x_1^3 + x_2^4x_1 + x_3^8x_1$. Then we have $f = x_1^3 + x_2^4x_1 + x_3^8x_1+x_2^4x_3^4 + x_4^2$ with $f_{add} = x_2^4x_3^4$, the reduced system of weights $(1/3,1/6,1/12,1/2)$ and the graph
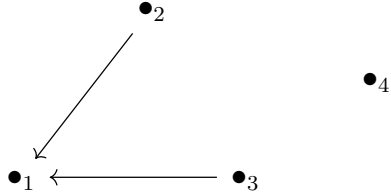
\[\begin{tikzcd}[sep=small]
	&& {\bullet_2} \\
	&&&&& {\bullet_4} \\
	\\
	{\bullet_1} &&& {\bullet_3}
	\arrow[from=1-3, to=4-1]
	\arrow[from=4-4, to=4-1]
\end{tikzcd}\]
Consider the group $G \cong \ZZ/2\ZZ \cong \langle g \rangle$ where $g(x_1,x_2,x_3,x_4) =(x_1,-x_2, x_3, -x_4)$ and thus we obtain $\bar{f} = x_1^3 + x_2^2x_1 + x_3^8x_1+x_2^2x_3^4 +x_4^2x_3$ with $\bar{f}_{add} =x_2^2x_4^4$, the reduced system of weights $(1/3,1/3,1/12,1/3)$ and the graph 
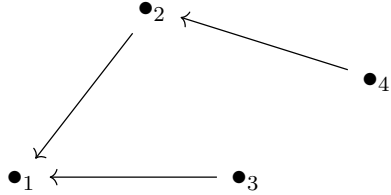
\[\begin{tikzcd}[sep=small]
	&& {\bullet_2} \\
	&&&&& {\bullet_4} \\
	\\
	{\bullet_1} &&& {\bullet_3}
	\arrow[from=1-3, to=4-1]
	\arrow[from=4-4, to=4-1]
	\arrow[from=2-6, to=1-3]
\end{tikzcd}\]
such that there is an orbifold equivalence $(f, \ZZ/2\ZZ) \sim (\bar{f}, \{\id\})$.
\end{example}
\section{Proof of Theorem \ref{main}}
We are going to show that $\bar{f}$ described in Theorem \ref{main} coincides with $\hat{f}_1$ obtained by crepant resolution. Firstly we prove that $\hat{f}_1$ is non-degenerate (i.e. it is quasihomogeneous and defines an isolated singularity at the origin) and has the properties described in Theorem \ref{main}. Secondly we show that $\hat{f}_2$ does not have any singularities. Then by Theorem \ref{equiv} the corresponded Landau-Ginzburg orbifolds $(f,G)$ and $(\hat{f_1}, \{\id\})$ will be orbifold equivalent that completes the proof. 

Let us rewrite $\hat{f}_{1}$ in the explicit form. Following decomposition for $f$ we have  $\hat{f}_{1}= (\hat{f}_\kappa)_{1} + (\hat{f}_{add})_{1}$. Recall that $f_0$ is a polynomial with loop with branches graph $\kappa_0$ described in Section 6. Then 
\begin{align}\label{exform}
(\hat{f}_\kappa)_{1} = f_{\kappa_0}(y, x_{2}, \dots,x_N) + z^2y
\end{align}

Since $G \cong \ZZ/2\ZZ$ is a group of symmetries, then the variable $x_1$ in each term of $f_{add}$ has an even powers. Namely, we have the explicit form of $f_{add}$: 
\begin{align*}
f_{add} = \sum_{\{J_r \in A_R\;|\;
1 \notin J_r \}} \varepsilon_{J_r}x_{s_1}^{b_{s_1}}x_{s_2}^{b_{s_2}} \dots x_{s_l}^{b_{s_l}} + \sum_{\{J_r \in A_R\;|\;
1 \in J_r\}} \varepsilon_{J_r}x_{1}^{2b_1}x_{s_2}^{b_{s_2}} \dots x_{s_l}^{b_{s_l}}
\end{align*}
for the admissible collection $A_R$ and $\varepsilon_{J_r} \in \CC^*$. And after the change of variables we obtain:
\begin{align*}
f_{add} =& \sum_{\{J_r \in A_R\;|\;
1 \notin J_r\}} \varepsilon_{J_r}x_{s_1}^{b_{s_1}}x_{s_2}^{b_{s_2}} \dots x_{s_l}^{b_{s_l}} + \sum_{\{J_r \in A_R\;|\;
1 \in J_r\}} \varepsilon_{J_r}y^{b_1}x_{s_2}^{b_{s_2}} \dots x_{s_l}^{b_{s_l}}
\end{align*}
\begin{proposition}
    Polynomial $\hat{f}_1$ is quasihomogeneous with the reduced system of weights $(2q_1, q_{2}, \dots, q_N, 1/2 - q_1)$, where $(q_1, \dots, q_N)$ are the weights of $f_0$.
\end{proposition}
\begin{proof}
The reduced weights are defined by the equations $a_iq_i + q_{\kappa(i)} = 1$ where $a_i$ are the fixed powers of the polynomial and $q_{\kappa(i)} = 0$ iff $\kappa(i) = i$ (in particular $q_{N+1} = 1/2$). Recall that $f$ has a set of powers $(2a_1, a_2, \dots, a_N, 2)$ and let $(\hat{a}_1, \hat{a}_2, \dots, \hat{a}_{N+1})$ be the new powers after a resolution and a restriction on the first chart. It is easy to see that if $s \neq 1,N+1$, then $a_s = \hat{a}_s$. Therefore, the equations $\hat{a}_sq_s + q_{\kappa(s)} = 1$ hold and it means that $\hat{q}_s = q_s$. Now if $s=1$, we have $\hat{a}_1 = 2a_1/2 = a_1$ and in this case the equation $\hat{a}_{1}(2q_1) + \hat{q}_{2} = a_1(2q_1) + q_{2} = 2a_1q_1 + q_{2} = 1$ holds from which we conclude that $\hat{q}_1 = 2q_1$. Similarly for $s = N+ 1$ we know that $\hat{a}_{N+1} = 2$ due to explicit form of $\hat{f}_{1}$, and we just have to solve the equation \begin{align*}
    \hat{a}_{N+1}\hat{q}_{N+1} + \hat{q}_{1} = 1 \\
    2\hat{q}_{N+1} + 2q_1 = 1 \\
    \hat{q}_{N+1} = \frac{1-2q_1}{2}
\end{align*} 
The condition $\sum_{s \in J_k} b_{s}q_{s} = 1$ holds since for each choice of $J_k$ we also have the changes only for $s=1$ and $\hat{b}_{k_1} = 2b_{k_1}/2= b_{k_1}$ which implies $2b_{k_1}q_1 = \hat{b}_{k_1}2q_1 = \hat{b}_{k_1}\hat{q}_1$.
\end{proof}
\begin{corollary}
    Since $\hat{f}_{1}$ is quasihomogeneous, it should be corresponded to some graph. By the explicit form (\ref{exform}), we obtain that it is the graph $\Gamma_{\bar{\kappa}} = \circ_{\Gamma_\kappa}(N+1, 1)$.
\end{corollary}

Since we proved that $\hat{f}_{1}$ is quasihomogeneous, the last statement that we need is the following:
\begin{proposition}
     All critical points of $\hat{f}$ are on the chart $U_{1}$. Moreover, $\hat{f}_{1}$ is non-degenerate.
\end{proposition}
\begin{proof}
Our aim is to prove that $0$ is an isolated solution of the system $\{d \hat{f}_{s} = 0\}$ only if $s=1$. Recall the explicit form of $\hat{f}_{1}$ and put $y = x_1$ and $z=x_{N+1}$:
\begin{align*}
  \hat{f}_{1} = (\hat{f}_\kappa)_{1} + (\hat{f}_{add})_{1} = f_{\kappa_0}(x_1,\dots,x_N) + x_{N+1}^2x_1 +& \sum_{\{J_r \in A_R\;|\;
1 \notin J_r\}} \varepsilon_{J_r}x_{s_1}^{b_{s_1}}x_{s_2}^{b_{s_2}} \dots x_{s_l}^{b_{s_l}} +\\ +& \sum_{\{J_r \in A_R\;|\;
1 \in J_r\}} \varepsilon_{J_r}x_{1}^{b_1}x_{s_2}^{b_{s_2}} \dots x_{s_l}^{b_{s_l}}
\end{align*}
It is not hard to note that if $s$ such that $1 < s < N+1$ and $s$ does not lie in any set from $A_R$, then
\[\frac{\partial \hat{f}_{1}}{\partial x_s} =  \frac{\partial f_{\kappa_0}}{\partial x_s} \]
So in this case $0$ is a solution of equations above since $f$ is non-degenerate.
If $s$ such that $1 < s < N+1$ and $s\in J_k$ for some $J_k \in A_R$, then the system of equations
\[\frac{\partial \hat{f}_1}{\partial x_s} = a_sx_s^{a_s-1}t_{\kappa(s)} + \sum_{w \in \kappa^{-1}(s)}x_w^{a_w} + \sum_{J_k \in A_R \wedge s \in J_k} \varepsilon_{J_k}b_{s}x_{s}^{b_{s}-1}x_{k_2}^{b_{k_2}} \dots x_{k_l}^{b_{k_l}} \]
obviously vanishes at $0$. And for $s = 1,N+1$ we also have $0$ as a solution:
\[
\frac{\partial \hat{f}_{1}}{\partial x_1} = a_1x_1^{a_1-1}x_{2} + x_{N+1}^2 + \sum_{\{J_r \in A_R\;|\;
1 \in J_r\}} \varepsilon_{J_r}2b_1x_{1}^{2b_1-1}x_{s_2}^{b_{s_2}} \dots x_{s_l}^{b_{s_l}} = 0
\]
\[
\frac{\partial \hat{f}_1}{\partial x_{N+1}} = x_1x_{N+1} = 0
\]

We have shown that $0$ is a solution of the equations above. Also note that the change of variables $\{t_1 = x_1, t_{N+1}^2 = x_1x_{N+1}^2, t_1t_{N+1}= x_1x_{N+1}\}$ is a diffeomorphism and since $f$ is non-degenerate then  $\hat{f}_1$ is also non-degenerate.

Let us turn to $\hat{f}_2$, rewrite it similarly to $\hat{f}_1$ and consider two equations:
\[
\frac{\partial \hat{f}_2}{\partial x_1} = 2a_1x_1^{2a_1-1}x_{N+1}^{a_{N+1}}x_{2} + \sum_{J_k \in A_R \wedge 1 \in J_k} \varepsilon_{J_k}2b_{k_1}x_1^{2b_{k_1}-1}x_{N+1}^{b_{k_1}}x_{k_2}^{b_{k_2}} \dots x_{k_l}^{b_{k_l}} = 0
\]
\[
\frac{\partial \hat{f}_2}{\partial x_{N+1}} = a_ix_1^{2a_1}x_{N+1}^{a_1-1}x_{2} + \sum_{J_k \in A_R \wedge 1 \in J_k} \varepsilon_{J_k}b_{k_1}x_1^{2b_{k_1}}x_{N+1}^{b_{k_1}-1}x_{k_2}^{b_{k_2}} \dots x_{k_l}^{b_{k_l}} + 1 = 0
\]
Obviously this equations does not have the solution at the origin. Moreover, LHS of the second equation does not vanish if $x_1=0$ or $x_{N+1}=0$. It means we can suppose that $x_1\neq 0$ and $x_{N+1} \neq 0$. Then we divide the first equation on $2x_1^{2a_1-1}x_{N+1}^{a_1}$ and the second one on $x_1^{2a_1}x_{N+1}^{a_1-1}$, and obtain the following:
\[
a_1x_{2} + \sum_{J_k \in A_R \wedge 1 \in J_k} \varepsilon_{J_k}b_{k_1}x_1^{2b_{k_1}-2a_1}x_{N+1}^{b_{k_1}-a_1}x_{k_2}^{b_{k_2}} \dots x_{k_l}^{b_{k_l}} = 0
\]
\[
a_1x_{2} + \sum_{J_k \in A_R \wedge 1 \in J_k} \varepsilon_{J_k} b_{k_1}x_1^{2b_{k_1}-2a_1}x_{N+1}^{b_{k_1}-a_1}x_{k_2}^{b_{k_2}} \dots x_{k_l}^{b_{k_l}} + x_1^{-2a_1}x_{N+1}^{1-a_1} = 0
\]
which implies $x_1^{-2a_1}x_{N+1}^{1-a_1} = 0$. Thus, the system $\{d\hat{f}_2 = 0\}$ does not have any solution what completes the proof.
\end{proof}
\begin{remark}
    Note that since the orbifold equivalence is an equivalence relation, we can expand our result on the case of $k$ isolated points and $G = \underbrace{\ZZ/2\ZZ\times \dots \times \ZZ/2\ZZ}_{k}$ consistently applying Theorem \ref{main}.
\end{remark}
\section{Isomorphism of algebras}
From the equivalence of categories we have an isomorphism of the corresponding Frobenius algebras. In this section we give this isomorhism explicitly.
 We associate with Landau-Ginzburg orbifold $(f,G)$ Hochschild cohomology of the category of $G$-equivariant matrix factorizations $\ccHH^*(MF_G(f))$. To work with this ring we construct algebra $\Jac(f,G)$ such that $\Jac(f,G) \cong \ccHH^*(MF_G(f)$ (see \cite{Sh} \cite{BTW1}, \cite{BT1} for details). To define $\Jac(f,G)$ let us introduce a vector space: 
\[ \Jac'(f,G)= \bigoplus_{g \in G} \Jac(f^g)\xi_g\]
where $\xi_g$ are formal generators associated with $g \in G$. We define multiplication in this vector space as follows (see \cite{Sh} for details).
\begin{definition} Let $\theta_1, \theta_2, \dots, \theta_N$ и $\partial_{\theta_1}, \partial_{\theta_2}, \dots, \partial_{\theta_N}$ be formal variables. Then the \emph{Clifford algebra $Cl_N$} is a factor-algebra of $\mathbb{C}[\theta_1, \theta_2, \dots, \theta_N, \partial_{\theta_1}, \partial_{\theta_2}, \dots, \partial_{\theta_N}]$ by the following relations:
\begin{align*}
    \theta_i\theta_j &= -\theta_j\theta_i \\  \partial_{\theta_i}\partial_{\theta_j} &= -\partial_{\theta_j}\partial_{\theta_i}\\
    \partial_{\theta_i}\theta_j &= \delta_{ij} - \theta_j\partial_{\theta_i}
\end{align*}
\end{definition}
For $I\subseteq\{1,\ldots,n\}$ we have the following notations:
\begin{align*}
\partial_{\theta_{I}}:=\prod_{i\in I}\partial_{\theta_i}, 
\quad {\theta_{I}}:=\prod_{i\in I}{\theta_i},
\end{align*}
where indices are written in a increasing order. We also introduce the following notations for the $Cl_N$-modules:  
\[
\CC[{\theta}] := Cl_N/_{Cl_N\langle_{\partial_{\theta_1}},\ldots,{\partial_{\theta_N}}\rangle}, \quad \CC[\partial_{\theta}] := Cl/_{Cl_N\langle{{\theta_1}},\ldots,{{\theta_N}}\rangle} 
\]  

\begin{definition} The map $\bigtriangledown_i^{x\to(x,y)}\colon\mathbb{C}[\textbf{x}]\to\mathbb{C}[\textbf{x},\textbf{y}]$ 
\begin{center}
$\bigtriangledown_i^{x\to (x,y)}(f(x)) = \frac{f(y_1, y_2, \dots, y_{i-1}, x_{i}, x_{i+1}, \dots, x_N) - f(y_1, y_2, \dots, y_{i}, x_{i+1}, x_{i+2}, \dots, x_N)}{x_i-y_i}$
\end{center}
is called $i$-th \emph{difference derivative} of the polynomial $f(x)$.
\end{definition}

Now we describe the structure constants of multiplication in $\Jac'(f,G)$. For $g \in G$ denote $I_g := \{i | g_i = 1 \}, \; I_g^c := \{1,\dots, N\} \; \backslash I_g$ and $d_g := |I_g^c|$. For each pair $(g,h) \in G\times G$ we define $d_{g,h}:=\frac{1}{2}(d_g +d_h - d_{gh})$, and we define $\sigma_{g,h}\in \Jac(f^{gh})$ as follows:
\begin{itemize}
\item
If $d_{g,h} \notin \ZZ_{\geq 0}$, then set $\sigma_{g,h} = 0$.
\item
If $d_{g,h} \in \ZZ_{\geq 0}$, then we define $\sigma_{g,h}$ as the coefficient before $\partial_{\theta_{I_{gh}^c}}$ in the expression expansion:
\begin{align*}
\frac{1}{d_{g, h} !} \Upsilon\left(\left(\left\lfloor\mathrm{H}_{f}(x, g(x), x)\right\rfloor_{g h}+\left\lfloor\mathrm{H}_{f, g}(x)\right\rfloor_{g h} \otimes 1+1 \otimes\left\lfloor\mathrm{H}_{f, h}(g(x))\right\rfloor_{g h}\right)^{d_{g, h}} \otimes \partial_{\theta_{I_{g}^c}} \otimes \partial_{\theta_{I_{h}^c}}\right)    
\end{align*}
where 
\begin{itemize}
\item[(1)] $\mathrm{H}_{f}(x, g(x), x)$ is an element of $\CC[\bx]\otimes\CC[\theta]^{\otimes2}$ and defined by
\begin{align*}
\rmH_{f}(\bx,\by,\bz):=\sum_{1\leq j\leq i\leq n} \nabla^{\by\to(\by,\bz)}_j\nabla^{\bx\to(\bx,\by)}_i(f)\,\theta_i\otimes \theta_j;
\end{align*}
\begin{align*}
\mathrm{H}_{f}(x, g(x), x) = \rmH_{f}(\bx,\by,\bz)|_{\{\by=g(\bx),\, \bz=\bx\}}
\end{align*}
\item[(2)]
$\rmH_{{f,g}}(\bx)$ is an element of module $\CC[\bx]\otimes \CC[\theta]$ given by the expression
\begin{eqnarray*}
\rmH_{{f,g}}(\bx):=\sum_{{i,j\in I_{g}^c,\,\,  j<i}}\frac{1}{1-g_j}\nabla^{\bx\to(\bx,\bx^g)}_j\nabla^{\bx\to(\bx,g(\bx))}_i(f)\,\theta_j\,\theta_i,
\end{eqnarray*}
where $\bx^g$ is defined by $(\bx^g)_i=x_i$ if $i\in I_g$ and $(\bx^g)_i=0$ if $i\in I_g^c$;
\item[(3)]
$\left\lfloor\rm- \right\rfloor_{gh}:\CC[\bx]\otimes V\longrightarrow \Jac(f^{gh})\otimes V$ for 
$V=\CC[\bx]\otimes \CC[\theta]^{\otimes 2}$ or $V=\CC[\bx]\otimes \CC[\theta]$
is a $\CC$-linear extension of the map  $\CC[\bx]\longrightarrow \Jac(f^{gh})$;
\item[(4)]
The degree $d_{g,h}$ is calculated with respect to the natural multiplication defined by $\CC[\bx]\otimes \CC[\theta] \otimes \CC[\theta]$;
\item[(5)]
$\Upsilon$ is a $\CC[\bx]$-linear extension of the map $\CC[{\theta}]^{\otimes2}\otimes \CC[\partial_{\theta}]^{\otimes2}\to\CC[\partial_{\theta}]$, defined as
\begin{align*}
 p_1(\theta)\otimes p_2(\theta)\otimes q_1(\partial_\theta)\otimes q_2(\partial_\theta)\mapsto(-1)^{|q_1||p_2|}p_1(q_1)\cdot p_2(q_2)
\end{align*}
where $p_i(q_i)$ is the action of $p_i(\theta)$ on $q_i(\partial_\theta)$ according to the multiplication structure in Clifford algebra defined above. 
\end{itemize}
\end{itemize}

Then we define the multiplication as follows:
\begin{align*}\label{eq: HH cup}
[\phi(\bx)]\xi_g\cup [\psi(\bx)]\xi_{h} 
=[\phi(\bx)\psi(\bx) \sigma_{g,h}]\xi_{gh},\quad \phi(\bx), \psi(\bx)\in \CC[\bx].
\end{align*}
We also endowing $\Jac'(f,G)$ with an action as follows. For $h \in G$ and $[\phi(x)] \in \Jac(f^g)$
\begin{align*}
 h=(h_1, \ldots, h_n): \quad [\phi(x)] \xi_g\mapsto \prod_{i\in I_g^c} h^{-1}_i\,\cdot  [\phi(h(x))]\xi_g, \;\;\; [\phi(x)] \in \Jac(f^g)   
\end{align*}
Then we define $\Jac(f,G)=(\Jac'(f,G))^G$. In particular, in Theorem \ref{main} we have 
\begin{align*}
    \Jac(\bar{f}, \{id\}) \cong \Jac(\bar{f}) \;\;\; \Jac(f,G) \cong (\Jac(f) \xi_{id})^G \oplus \Jac(f^g)\xi_g
\end{align*}
\begin{proposition}\label{HHJ}
    There is an isomorphism of algebras $\psi\colon\Jac(\bar{f})\DistTo\Jac(f,G)$. In particular,
    \begin{center}
$\psi([x_i]) = \begin{cases} [x_i]\xi_{id}$, if $i\neq 1, N+1\\
 [x_1^2]\xi_{id}$, if $i = 1\\
\xi_g$, if $i = N+1
\end{cases}$
\end{center}
\end{proposition}\textit{}
\begin{proof}
 Let us show that $\Jac(\bar{f})$ has a decomposition into sum of 2 vector spaces:
\begin{align*}
    \Jac(\bar{f}) = \mathcal{B}_1\oplus\mathcal{B}_2
\end{align*}
such that $\mathcal{B}_1 \cong (\Jac(f))^G$ as algebras and $\mathcal{B}_2 \cong \Jac(f^g)x_{N+1}$ as vector spaces. 
We construct basis of $\mathcal{B}_1$ as follows. Let $[x_1^{\alpha_1}x_2^{\alpha_2}\dots x_{N+1}^{\alpha_{N+1}}]$ be a basis element of $\Jac(\bar{f})$. Note that if $\alpha_1 \neq 0$ and $\alpha_{N+1} = 0$ (i.e. this element does not depend on variable $x_{N+1}$), then $[x_1^{2\alpha_1}x_2^{\alpha_2}\dots x_{N}^{\alpha_{N}}]$ could be taken as a basis element of $\Jac(f)$ by construction of $\bar{f}$. Moreover, the multiplication of such elements $[x_1^{2\alpha_1}x_2^{\alpha_2}\dots x_{N}^{\alpha_{N}}]$ in $\Jac(\bar{f})$ is coincide with multiplication in $(\Jac(f)\xi_{id})^G$ due to the invariance under $G$-action, thus we obtain the basis of algebra $(\Jac(f)\xi_{id})^G$. 

To construct basis of $\mathcal{B}_2$, we consider the basis elements of $\Jac(\bar{f})$ such that $\alpha_{N+1} \neq 0$. Then $\alpha_{N+1}$ should be  equal to $1$ and $\alpha_1=0$ by the relation $[\frac{\partial \bar{f}}{\partial x_{N+1}}] = [2x_1x_{N+1}] = [0]$. Thus we have $[x_1^{\alpha_1}x_2^{\alpha_2}\dots x_{N+1}^{\alpha_{N+1}}] =[x_2^{\alpha_2}\dots x_{N}^{\alpha_{N}}][x_{N+1}]$ with the first factor lying in $\Jac(f^g)$, which gives us the necessary isomorphism.

The last thing that we have to prove is that $[x_{N+1}]^2 = (\xi_g)^2$. We calculate $\sigma_{g,g^{-1}}$ following the formula above. Note that
\begin{align*}
H_{f,g}(x) = \frac{1}{2}\nabla^{\bx\to(\bx,\bx^g)}_1\nabla^{\bx\to(\bx,g(\bx))}_{N+1}(f)\,\theta_1\,\theta_{N+1} = 0
\end{align*}
since $f$ does not have a summand containing both variables $x_1$ and $x_{N+1}$. Similarly $H_{f, g^{-1}}(g(x))=0$. Then $\sigma_{g,g^{-1}}$ is the coefficient of $1$ in the expression
\begin{align*}
    \frac{1}{2}\Upsilon([\rmH_f(x,g(x),x)]^2\otimes\partial_{\theta_{1}}\partial_{\theta_{1}}\otimes\partial_{\theta_{N+1}}\partial_{\theta_{N+1}})
\end{align*}
    Now let $A_{km} \in \CC[\textbf{x}]$ be the polynomials such that $\rmH_f$ is the following sum:
\begin{align*}
    \rmH_f(x,g(x),x)= \sum_{ (i,j) \neq (1,1) \; (i,j) \neq (N+1,N+1)} A_{ij}(x)\theta_i \otimes \theta_j + A_{11}(x)\theta_1 \otimes \theta_1 + A_{N+1,N+1}(x)\theta_{N+1}\otimes\theta_{N+1}
\end{align*}
Then
\[[\rmH_f]^2= 2[A_{11}A_{N+1,N+1}]\theta_1\theta_{N+1}\otimes\theta_1\theta_{N+1} + \sum_{(i,j,k,l) \neq (1,N+1,1,N+1)}[\Tilde{A}_{i,j,k,l}]\theta_i\theta_j\otimes\theta_k\theta_l\]
where we $\Tilde{A}_{i,j,k,l} \in \CC[\textbf{x}]$ are the polynomials obtaining by multiplication in $\CC[\bx]\otimes \CC[\theta] \otimes \CC[\theta]$. Consequently, we conclude that
\[\sigma_{g,g^{-1}} = -[A_{11}A_{N+1,N+1}]\] 
and by the exact calculations obtain
\[[A_{N+1,N+1}] = [1]\]
\[A_{11} = [a_1x_1^{2(a_1-1)}x_2 + \sum_{J_k \in A_R \wedge 1 \in J_k} \varepsilon_{J_k}b_{k_1}x_1^{2(b_{k_1}-1)}x_{k_2}^{b_{k_2}} \dots x_{k_l}^{b_{k_l}}]\]
which implies the following relation in $\Jac(f,G)$:
\[(\xi_g)^2 = -[a_1x_1^{2(a_1-1)}x_2 + \sum_{J_k \in A_R \wedge 1 \in J_k} \varepsilon_{J_k}b_{k_1}x_1^{2(b_{k_1}-1)}x_{k_2}^{b_{k_2}} \dots x_{k_l}^{b_{k_l}}]\xi_{id}\]
Now we consider $\Jac(\bar{f})$ and write the relation given by the partial derivative:
\[[\frac{\partial f}{\partial x_1}] = [x_{N+1}^2 + a_1x_1^{a_1-1}x_2+ \sum_{J_k \in A_R \wedge 1 \in J_k} \varepsilon_{J_k}b_{k_1}x_1^{b_{k_1}-1}x_{k_2}^{b_{k_2}} \dots x_{k_l}^{b_{k_l}}] = [0]\]
\[
[x_{N+1}]^2 = - [ a_1x_1^{a_1-1}x_2 + \sum_{J_k \in A_R \wedge 1 \in J_k} \varepsilon_{J_k}b_{k_1}x_1^{b_{k_1}-1}x_{k_2}^{b_{k_2}} \dots x_{k_l}^{b_{k_l}}]
\]
from which we obtain the claim.
\end{proof}


\end{document}